\documentclass[preprint,12pt]{elsarticle}
\usepackage{amsmath,amssymb,amsthm}
\usepackage{cases}
\usepackage{xcolor}
\usepackage{float}
\usepackage{hyperref}

\hypersetup{
 hidelinks,
 pdftitle={Improved Global Boundedness Criteria via Refined Bootstrap Estimates for a Quasilinear Chemotaxis-May-Nowak System},
 pdfauthor={Xueli Bai, Fang Li, Suying Liu, and Xuxuan Luo},
 pdfsubject={Global boundedness in a quasilinear chemotaxis-May-Nowak system},
 pdfkeywords={Chemotaxis-May-Nowak system; nonlinear diffusion; global boundedness; parabolic-elliptic-parabolic system; fully parabolic system}
}

\numberwithin{equation}{section}

\newtheorem{thm}{Theorem}[section]

\newtheorem{lem}[thm]{Lemma}
\newtheorem{prop}[thm]{Proposition}

\journal{Journal of Differential Equations}
\biboptions{sort&compress}
\begin{document}

\begin{frontmatter}

\title{Improved Global Boundedness Criteria via Refined Bootstrap
Estimates for a Quasilinear Chemotaxis--May--Nowak System}

\author[nwpu]{Xueli Bai}
\ead{xlbai2015@nwpu.edu.cn}
\author[sysu]{Fang Li}
\ead{lifang55@mail.sysu.edu.cn}
\author[nwpu]{Suying Liu\corref{cor1}}
\ead{suyingliu@nwpu.edu.cn}
\author[sysu]{Xuxuan Luo}
\ead{luoxx35@mail2.sysu.edu.cn}
\cortext[cor1]{Corresponding author.}

\address[nwpu]{School of Mathematics and Statistics, Northwestern
Polytechnical University, 1 Dongxiang Road, Chang'an District, Xi'an,
Shaanxi 710129, P.R. China}
\address[sysu]{School of Mathematics, Sun Yat-sen University, 135 Xingang
West Road, Haizhu District, Guangzhou 510275, P.R. China}

\begin{abstract}
We study a quasilinear chemotaxis--May--Nowak system with nonlinear diffusion
growing at least as $s^{m-1}$ at high densities and derive improved criteria
for global existence and uniform boundedness. For the
parabolic--elliptic--parabolic system, we establish global boundedness when
$m>3-\frac{4}{n}$ for $n\in\{2,3\}$ and when $m>1+\frac{n}{4}$ for $n\ge4$.
For the fully parabolic system, the same conclusion holds under
$m>1+\frac{n}{4}$ for $n\ge4$. The proof relies on two refined estimates.
Weighted space--time control of the nonlinear infection source yields the
low-order $L^q$-bounds for the free virus density needed to initiate the
bootstrap in the parabolic--elliptic--parabolic case. Subsequently, preserving
the product structure of this source in the interpolation argument allows
the chemotactic contribution to be absorbed more efficiently by the
nonlinear diffusion. These estimates improve the previously known sufficient
conditions for global boundedness.
\end{abstract}


\end{frontmatter}

\section{Introduction}

Mathematical models for within-host virus dynamics provide a natural framework
for describing the coupled evolution of uninfected cells, infected cells, and
free virus particles. When spatial heterogeneity and directed cell movement
are incorporated, such models can capture localized infection patterns that
are beyond the scope of spatially homogeneous systems. Their analysis is
therefore important for understanding how diffusive and taxis-driven
mechanisms interact to determine the spatio-temporal behavior of viral
infections.

This paper is concerned with the following quasilinear
chemotaxis--May--Nowak system for virus infection:
\begin{numcases}{}
 u_t=\nabla\cdot(D(u)\nabla u)-\chi\nabla\cdot(u\nabla v)-u-uw+\varphi,
 & $x\in\Omega,\ t>0$,\nonumber\\
 \kappa v_t=\Delta v-v+uw,
 & $x\in\Omega,\ t>0$,\nonumber\\
 w_t=\Delta w-w+v,
 & $x\in\Omega,\ t>0$,\nonumber\\
 \dfrac{\partial u}{\partial\nu}
 =\dfrac{\partial v}{\partial\nu}
 =\dfrac{\partial w}{\partial\nu}=0,
 & $x\in\partial\Omega,\ t>0$,\nonumber\\
 u(x,0)=u_0(x),\quad
 \kappa v(x,0)=\kappa v_0(x),
 & $x\in\Omega$,\nonumber\\
 w(x,0)=w_0(x),\qquad x\in\Omega\label{sys}
\end{numcases}
in a bounded smooth domain $\Omega\subset\mathbb R^n$, $n\ge2$.
Here $u$, $v$, and $w$ denote the densities of uninfected cells, infected
cells, and free virus particles, respectively. The infection process is
represented by the nonlinear coupling $uw$, while the taxis term
$-\chi\nabla\cdot(u\nabla v)$ describes directed movement of uninfected cells
toward regions with higher densities of infected cells. The nonlinear
diffusion $\nabla\cdot(D(u)\nabla u)$ accounts for density-dependent dispersal
of uninfected cells, and the nonnegative function $\varphi$ models their
external supply. The parameter $\chi>0$ is the chemotactic sensitivity, and
$\kappa=1$ corresponds to the fully parabolic system, whereas $\kappa=0$
gives the parabolic--elliptic--parabolic system.

System \eqref{sys} originates from spatial extensions of the classical May--Nowak model
for within-host viral dynamics, which describes the interaction of uninfected
cells, infected cells, and free virions \cite{maynowak2, maynowak}.
{While its spatially homogeneous form has provided important
insight into viral persistence and treatment \cite{may3,relevantresults4},
spatial heterogeneity and directed cell movement can play an essential role in
the formation of highly concentrated infection patterns.} {To
account for these effects,} Stancevic et al.
\cite{chemotaxismodel} incorporated diffusion and chemotaxis into the May--Nowak
dynamics. Further developments of this model have led to various
chemotaxis--May--Nowak systems, including the quasilinear formulation with
nonlinear diffusion considered in \eqref{sys}; see, in particular,
\cite{BellomoPainterTaoWinkler2019,TaoWinkler2021Hotspots,wang2025}.

A first fundamental step in understanding the dynamics of \eqref{sys} is to
determine whether classical solutions exist globally in time or may develop
singularities in finite time. These two possibilities reflect the competition
between diffusive dispersion and taxis-driven aggregation. In the context of
virus dynamics, finite-time blow-up can be viewed as an idealized mathematical
description of singular infection hotspots, whereas global boundedness excludes
such unbounded concentration. The present paper focuses on the latter issue and
seeks sufficient conditions ensuring that classical solutions of \eqref{sys}
exist globally and remain uniformly bounded.

More specifically, we focus on the stabilizing effect of nonlinear diffusion.
For the prototype $D(u)=(1+u)^{m-1}$, the exponent $m$ measures the strength
of density-dependent dispersal at high cell densities. Since $D(u)\sim u^{m-1}$ as $u\to\infty$,
increasing $m$ strengthens
diffusion in high-density regions and thereby counteracts the
concentration induced by chemotaxis. The central question is therefore how large $m$ must be
for nonlinear diffusion to dominate the taxis mechanism and preclude
unbounded aggregation.

Throughout this paper, we assume that
\begin{numcases}{}
 D\in C^2([0,\infty)),\quad D(s)>0,\quad
 D(s)\ge k_Ds^{m-1}
 & for all $s\ge0$,\nonumber\\
 0\le\varphi\in C^0(\overline\Omega\times[0,\infty)),\quad
 \|\varphi\|_{L^\infty(\Omega\times(0,\infty))}\le \varphi_*.
 & \label{eq:basic_assumptions}
\end{numcases}
with some $k_D>0$. The strict positivity of $D$ is imposed because we consider
classical solutions. For $\kappa=0$, the initial data satisfy
\begin{numcases}{}
 u_0,w_0\in W^{2,\infty}(\Omega),
 & \nonumber\\
 u_0,w_0\ge0,\quad u_0\not\equiv0,\quad w_0\not\equiv0
 & in $\overline\Omega$,\nonumber\\
 \dfrac{\partial u_0}{\partial\nu}
 =\dfrac{\partial w_0}{\partial\nu}=0
 & on $\partial\Omega$.\label{eq1.2}
\end{numcases}
whereas for $\kappa=1$ we assume
\begin{numcases}{}
 u_0,v_0,w_0\in W^{2,\infty}(\Omega),
 & \nonumber\\
 u_0,v_0,w_0\ge0,\quad
 u_0\not\equiv0,\quad v_0\not\equiv0,\quad w_0\not\equiv0
 & in $\overline\Omega$,\nonumber\\
 \dfrac{\partial u_0}{\partial\nu}
 =\dfrac{\partial v_0}{\partial\nu}
 =\dfrac{\partial w_0}{\partial\nu}=0
 & on $\partial\Omega$.\label{eq1.3}
\end{numcases}

Questions of global existence, boundedness, and blow-up arising from the
competition between diffusion and chemotactic aggregation have been studied
extensively for Keller--Segel type systems; see the surveys
\cite{Horstmann2003,Horstmann2004,HillenPainter2009,Bellomo2015,Bellomo2022}
and the references therein. For chemotaxis--May--Nowak systems, Bellomo,
Painter, Tao and Winkler \cite{BellomoPainterTaoWinkler2019} investigated the
occurrence and suppression of taxis-driven instabilities in a fully
parabolic setting with linear diffusion, while Tao and Winkler
\cite{TaoWinkler2021Hotspots} established finite-time blow-up in a generalized
parabolic--elliptic--parabolic system, thereby rigorously confirming the
possibility of singular hotspot formation. Global boundedness has also been
obtained through regularizing mechanisms such as saturation of the
chemotactic sensitivity \cite{achemotaxissystem}.

For the quasilinear model \eqref{sys}, Wang and Wang
\cite{wang2025} initiated the study of global boundedness and blow-up
under nonlinear diffusion, and Zheng and Wang \cite{zheng2026}
subsequently improved several sufficient conditions for boundedness.
Building on these works, we develop refined $L^p$-estimates and obtain
further improvements of the known sufficient conditions for global
boundedness. {The parabolic--elliptic--parabolic and fully
parabolic cases require different estimates and are treated separately.} Our two main results
are as follows.

\begin{thm}\label{main_theorem_kappa=0}
Let $\kappa=0$ and $n\ge2$. Assume \eqref{eq:basic_assumptions} and
\eqref{eq1.2}, and suppose that
\begin{equation*}
 m>
 \begin{cases}
  1+\dfrac n4, & n\ge4,\\[3mm]
  3-\dfrac4n, & n\in\{2,3\}.
 \end{cases}
\end{equation*}
Then the solution of \eqref{sys} is global and there exists $C>0$ such that
$$
 \|u(\cdot,t)\|_{L^\infty(\Omega)}
 +\|v(\cdot,t)\|_{W^{1,\infty}(\Omega)}
 +\|w(\cdot,t)\|_{L^\infty(\Omega)}
 \le C
 \qquad\text{for all }t>0.
$$
\end{thm}

\begin{thm}\label{main_theorem_kappa=1}
Let $\kappa=1$ and $n\ge4$. Assume \eqref{eq:basic_assumptions} and
\eqref{eq1.3}, and suppose that $m>1+\frac n4$. Then the solution of
\eqref{sys} is global
and there exists $C>0$ such that
$$
 \|u(\cdot,t)\|_{L^\infty(\Omega)}
 +\|v(\cdot,t)\|_{W^{1,\infty}(\Omega)}
 +\|w(\cdot,t)\|_{L^\infty(\Omega)}
 \le C
 \qquad\text{for all }t>0.
$$
\end{thm}

The sufficient conditions obtained here are compared with those of
Wang and Wang \cite{wang2025} and Zheng and Wang \cite{zheng2026} in
Tables~\ref{tab:boundedness-kappa0} and~\ref{tab:boundedness-kappa1}, where
{the parabolic--elliptic--parabolic and fully parabolic cases
are listed separately.}

\begin{table}[H]
 \centering
 \small
 \renewcommand{\arraystretch}{1.25}
 \begin{tabular}{c|c|c|c}
  \hline
  Dimension & Wang--Wang~\cite{wang2025}
  & Zheng--Wang~\cite{zheng2026} & Present work \\
  \hline
  $n=2$ & $m>\frac32$ & $m>\frac32$ & $m>1$ \\
  \hline
  $n=3$ & $m>\frac{17}{6}$ & $m>\frac{25}{12}$ & $m>\frac53$ \\
  \hline
  $n\ge4$ & $m>2+\frac n2-\frac2n$
  & $m>2+\frac n4-\frac2n$ & $m>1+\frac n4$ \\
  \hline
 \end{tabular}
 \caption{Comparison of sufficient conditions for global boundedness in the
 {parabolic--elliptic--parabolic case} $\kappa=0$ in the
 dimensions covered by Theorem~\ref{main_theorem_kappa=0}.}
 \label{tab:boundedness-kappa0}
\end{table}

\begin{table}[H]
 \centering
 \small
 \renewcommand{\arraystretch}{1.25}
 \begin{tabular}{c|c|c|c}
  \hline
  Dimension & Wang--Wang~\cite{wang2025}
  & Zheng--Wang~\cite{zheng2026} & Present work \\
  \hline
  $n=4$ & $m>\frac52$ & $m>2$ & $m>2$ \\
  \hline
  $n\ge5$ & $m>1+\frac n2-\frac2n$
  & $m>\frac32+\frac n4-\frac2n$ & $m>1+\frac n4$ \\
  \hline
 \end{tabular}
 \caption{Comparison of sufficient conditions for global boundedness in the fully
 parabolic case $\kappa=1$ in the dimensions covered by Theorem~\ref{main_theorem_kappa=1}.}
 \label{tab:boundedness-kappa1}
\end{table}

{We first consider the parabolic--elliptic--parabolic case
$\kappa=0$.} As shown in
Table~\ref{tab:boundedness-kappa0}, Theorem~\ref{main_theorem_kappa=0}
improves the previously known boundedness criteria in every dimension
$n\ge2$. Most notably, in two dimensions the condition $m>1$ is sharp within
the class of power-law diffusions, in view of the available radial blow-up
results on balls:
our result guarantees global boundedness for every $m>1$, whereas blow-up may
occur at $m=1$ \cite{TaoWinkler2021Hotspots} and in the sublinear regime
$m<1$ \cite{wang2025}. In dimension $n=3$, the condition $m>\frac53$ also
improves the previous thresholds, while for $n\ge4$ the sufficient condition
is reduced to $m>1+\frac n4$.

We next turn to the fully parabolic case $\kappa=1$. As indicated in
Table~\ref{tab:boundedness-kappa1}, for $n=4$ the condition
$m>1+\frac n4=2$ coincides with the best previously known criterion of
Zheng and Wang \cite{zheng2026}, while improving that of Wang and Wang
\cite{wang2025}. For every $n\ge5$, the condition $m>1+\frac n4$ improves
both earlier criteria. In particular, our results yield the common sufficient condition
\[
 m>1+\frac n4
\]
{for both the parabolic--elliptic--parabolic and fully
parabolic systems when $n\ge4$.}

The boundedness proof follows the standard bootstrap strategy: we first
derive a uniform low-order $L^q$-estimate for $w$ and then use it to obtain
uniform $L^p$-bounds for $u$ for every finite $p>1$. The remaining
$L^\infty$-bounds and global extension follow from standard iteration,
regularity, and extensibility arguments. The improvement in the present work
lies in how the first two stages are carried out. A weighted space--time
argument recovers the required control of $w$ in the
{parabolic--elliptic--parabolic case},
while preserving the product structure of $uw$ in the subsequent
$L^p$-estimates allows the chemotactic term to be absorbed more efficiently
by the nonlinear diffusion. Together, these sequential refinements improve
the exponent balance and yield sharper sufficient conditions. We now explain
the two stages in detail.

We first establish a uniform low-order $L^q$-estimate for $w$, valid for
every $q$ satisfying
\[
 1\le q<\infty \quad\text{if } n=2,
 \qquad
 1\le q<\frac{n}{n-2}\quad\text{if }n\ge3.
\]
This estimate provides the starting point for controlling the nonlinear
source $uw$ and subsequently deriving higher-order bounds for $u$. The basic
mass estimate takes the form
\begin{equation}\label{eq:mass_intro}
 \int_\Omega u(\cdot,t)
 +\kappa\int_\Omega v(\cdot,t)
 +\int_\Omega w(\cdot,t)\le C.
\end{equation}
For $\kappa=1$, the pointwise-in-time $L^1$-bound for $v$ furnished by
\eqref{eq:mass_intro}, together with the standard $L^1$--$L^q$ smoothing
estimate for the $w$-equation, immediately gives the required bound for $w$.
When $\kappa=0$, however, the term involving $v$ is absent from
\eqref{eq:mass_intro}, so this direct argument is unavailable. We instead use
the integral identity for the $u$-equation to obtain weighted-in-time control
of the nonlinear source $uw$. The Brezis--Strauss elliptic estimate
\cite{BSeq} transfers this information to $v$, and the $w$-equation then
converts it into the desired pointwise $L^q$-bound for $w$. Thus, the same
range of low-order estimates is recovered despite the absence of pointwise
$L^1$-control of $v$, providing the starting point for the subsequent
bootstrap in the {parabolic--elliptic--parabolic case}.

We next turn to the second estimate, namely the derivation of uniform
$L^p$-bounds for $u$ for arbitrary finite $p>1$. After testing the
$u$-equation by $u^{p-1}$, the crucial task is to control the term
\[
 \int_\Omega u^p|\Delta v|,
\]
in a form that can be absorbed by the nonlinear diffusion. This is
complicated by the fact that, unlike the usual linear signal source $u$ in
classical Keller--Segel systems, the source in the $v$-equation is the
nonlinear product $uw$. When $\kappa=0$, the elliptic equation for $v$ allows
the chemotactic term to be estimated pointwise in time, whereas when
$\kappa=1$, it is treated in weighted time integrals through maximal Sobolev
regularity. In both cases, we retain the product structure of $uw$ while
choosing the H\"older and interpolation exponents, and apply Young's
inequality only after the relevant factors have been combined. This preserves
the exponent balance of the product and allows the chemotactic term to be
absorbed more efficiently by the nonlinear diffusion. Together with the
low-order estimate for $w$, this yields the sufficient conditions stated in
Theorems~\ref{main_theorem_kappa=0} and
\ref{main_theorem_kappa=1}.

The remainder of the paper is organized as follows. Section 2 collects the
maximal-regularity, semigroup, interpolation, and elliptic estimates used
below. Section 3 treats the parabolic--elliptic--parabolic case $\kappa=0$,
and Section 4 considers the fully parabolic case $\kappa=1$.

\section{Preliminaries}

First of all, we present the known local solvability,
extensibility property and $L^1$ estimates used throughout the paper for the convenience of readers.
See \cite[Lemma 1.1]{wang2025} for example.

\begin{lem}[Local existence and extensibility]\label{global_bounded_lemma}
Let $\kappa\in\{0,1\}$ and $n\ge2$. Assume \eqref{eq:basic_assumptions}, and
suppose that the initial data satisfy \eqref{eq1.2} when $\kappa=0$ and
\eqref{eq1.3} when $\kappa=1$. Then there exist
$T_{\max}\in(0,\infty]$ and a nonnegative classical solution $(u,v,w)$ of
\eqref{sys} on $\Omega\times(0,T_{\max})$ such that
$$
 u,w\in C^0\bigl(\overline\Omega\times[0,T_{\max})\bigr)
 \cap C^{2,1}\bigl(\overline\Omega\times(0,T_{\max})\bigr),
$$
and
$$
 v\in
 \begin{cases}
 C^{2,0}\bigl(\overline\Omega\times(0,T_{\max})\bigr),&\kappa=0,\\
 C^0\bigl(\overline\Omega\times[0,T_{\max})\bigr)
 \cap C^{2,1}\bigl(\overline\Omega\times(0,T_{\max})\bigr),&\kappa=1.
 \end{cases}
$$
Moreover, $u,v,w>0$ in $\overline\Omega\times(0,T_{\max})$. If
$T_{\max}<\infty$, then
$$
 \limsup_{t\nearrow T_{\max}}
 \left(
 \|u(\cdot,t)\|_{L^\infty(\Omega)}
 +\|w(\cdot,t)\|_{L^\infty(\Omega)}
 \right)=\infty.
$$
\end{lem}

The following lemma provides uniform mass bounds for the solutions of
\eqref{sys} in both cases $\kappa=0$ and $\kappa=1$.

\begin{lem}[Mass bounds for $\kappa=0,1$]\label{lem-mass}

Let $\kappa\in\{0,1\}$ and $n\ge2$. Then there exists $C>0$ such that
$$
 \int_\Omega u(x,t)\,dx
 +\kappa\int_\Omega v(x,t)\,dx
 +\int_\Omega w(x,t)\,dx
 \le C
 \qquad\text{for all }t\in(0,T_{\max}).
$$
\end{lem}

\begin{proof}
The proof is omitted and can be found in \cite{wang2025}.
\end{proof}

Secondly, we prepare some known results on the maximal Sobolev regularity and the smoothing properties of the Neumann heat semigroup as well as the Gagliardo-Nirenberg inequality, which are useful in the proof of Theorems \ref{main_theorem_kappa=0} and \ref{main_theorem_kappa=1}.

\begin{lem}[Weighted maximal Sobolev regularity, \cite{MSR}]\label{max_sobolev_regularity}
Let $q,r\in(1,\infty)$, $\delta\in(0,1)$, and $T>0$. Suppose that $z$ solves
$$
\begin{cases}
 z_t=\Delta z-z+f,&x\in\Omega,\ t\in(0,T),\\
 \dfrac{\partial z}{\partial\nu}=0,&x\in\partial\Omega,\ t\in(0,T),\\
 z(\cdot,0)=z_0,&x\in\Omega,
\end{cases}
$$
where $f\in C(\overline\Omega\times(0,T))$ and
$z_0\in W^{2,\infty}(\Omega)$ satisfies the Neumann compatibility condition.
Then there exists $C=C(q,r,\delta,\Omega)>0$ such that, for every $t\in(0,T)$,
$$
 \int_0^t e^{-\delta r(t-s)}\|\Delta z(\cdot,s)\|_{L^q}^r\,ds
 \le C\int_0^t e^{-\delta r(t-s)}\|f(\cdot,s)\|_{L^q}^r\,ds
 +C\|z_0\|_{W^{2,q}}^r.
$$
\end{lem}


\begin{lem}[$L^1$--$L^q$ parabolic and elliptic smoothing]\label{lem-regularity}
Let $\Omega\subset\mathbb{R}^n$, $n\ge2$, be a bounded domain with smooth
boundary, and let
$$
 1\le q<\infty \quad\text{if } n=2,
 \qquad
 1\le q<\frac{n}{n-2}\quad\text{if }n\ge3.
$$
The following statements hold.

\begin{enumerate}
\item[\rm(i)]
Suppose that $z$ solves
$$
\begin{cases}
 z_t=\Delta z-z+f,
     &x\in\Omega,\quad t\in(0,T),\\
 \dfrac{\partial z}{\partial\nu}=0,
     &x\in\partial\Omega,\quad t\in(0,T),\\
 z(\cdot,0)=z_0,
     &x\in\Omega,
\end{cases}
$$
where $z_0\in L^q(\Omega)$ and
$$
 \sup_{t\in(0,T)}\|f(\cdot,t)\|_{L^1(\Omega)}\le M.
$$
Then there exists $C=C(q,\Omega)>0$, independent of $T$, such that
\begin{equation}\label{eq:parabolic_L1_Lq}
 \sup_{t\in(0,T)}\|z(\cdot,t)\|_{L^q(\Omega)}
 \le C\bigl(M+\|z_0\|_{L^q(\Omega)}\bigr).
\end{equation}

\item[\rm(ii)]
Suppose that
$$
 -\Delta z+z=f\quad\text{in }\Omega,
 \qquad
 \frac{\partial z}{\partial\nu}=0\quad\text{on }\partial\Omega,
$$
with $f\in L^1(\Omega)$. Then
\begin{equation}\label{eq:elliptic_L1_Lq}
 \|z\|_{L^q(\Omega)}
 \le C(q,\Omega)\|f\|_{L^1(\Omega)}.
\end{equation}
\end{enumerate}
\end{lem}

\begin{proof}
For (i), the variation-of-constants formula and the standard
$L^1$--$L^q$ smoothing estimate for the Neumann heat semigroup give
$$
 \|z(\cdot,t)\|_{L^q}
 \le e^{-t}\|z_0\|_{L^q}
 +C\int_0^t
 e^{-(t-s)}
 \left(1+(t-s)^{-\frac n2(1-\frac1q)}\right)
 \|f(\cdot,s)\|_{L^1}\,ds.
$$
The restriction on $q$ ensures that
$$
 \frac n2\left(1-\frac1q\right)<1,
$$
so the time convolution is uniformly bounded and
\eqref{eq:parabolic_L1_Lq} follows.

For (ii), the Brezis--Strauss estimate
\cite[Lemma~23]{BSeq} yields
$$
 \|z\|_{W^{1,r}(\Omega)}
 \le C_r\|f\|_{L^1(\Omega)}
 \qquad\text{for every }1\le r<\frac{n}{n-1}.
$$
Choosing $r$ sufficiently close to $\frac{n}{n-1}$ and using the Sobolev
embedding theorem gives \eqref{eq:elliptic_L1_Lq}.
\end{proof}

{
We shall repeatedly use the following Gagliardo--Nirenberg inequalities in
the sequel. For completeness, we provide a brief justification together
with precise references.}

\begin{lem}[Gagliardo--Nirenberg inequalities]\label{lem:GN}
The following two forms will be used below.
\begin{enumerate}
\item[\rm(i)]
Let $f\ge0$, $f\in W^{1,2}(\Omega)$, and $0<q<r<\infty$. Suppose that
$r\le\frac{2n}{n-2}$ when $n\ge3$, with arbitrary finite $r$ when $n=2$,
and define
$$
 a=\frac{\frac1q-\frac1r}{\frac1q+\frac1n-\frac12}\in[0,1].
$$
Then
$$
 \|f\|_{L^r}
 \le C\|\nabla f\|_{L^2}^{a}\|f\|_{L^q}^{1-a}
 +C\|f\|_{L^q}.
$$

\item[\rm(ii)]
Let $p_0,\rho,\gamma>1$, let $f\in W^{2,\rho}(\Omega)$ satisfy
$\frac{\partial f}{\partial\nu}=0$ on $\partial\Omega$, and suppose that
there exists $\alpha\in(0,1)$ such that
$$
 \frac1\gamma
 =\alpha\left(\frac1\rho-\frac2n\right)
 +\frac{1-\alpha}{p_0}.
$$
Then
$$
 \|f\|_{L^\gamma}
 \le C\|\Delta f\|_{L^\rho}^{\alpha}
 \|f\|_{L^{p_0}}^{1-\alpha}+C\|f\|_{L^{p_0}}.
$$
\end{enumerate}
\end{lem}

{
\begin{proof}
For (i), the subcritical case follows from
\cite[Appendix, Proposition~A.1]{DalPassoGiacomelliShishkov2001} by taking
the derivative order to be one and the derivative exponent to be two;
in particular, that result allows the lower integrability exponent $q$ to
be less than one. When $n\ge3$ and $r=\frac{2n}{n-2}$, the Sobolev
embedding gives
$$
 \|f\|_{L^{\frac{2n}{n-2}}}
 \le C\bigl(\|\nabla f\|_{L^2}+\|f\|_{L^2}\bigr)
 \le C\|\nabla f\|_{L^2}+C\|f\|_{L^q}.
$$
Indeed, if $q<2$, the second inequality follows by applying the already
established subcritical estimate with target exponent $2$ and then Young's
inequality, whereas if $q\ge2$, it follows directly from the boundedness of
$\Omega$. This proves the critical endpoint. See also
\cite[Lemma~2.2]{Winkler2002Critical} for a direct treatment of subunit
integrability exponents.

For (ii), the standard second-order Gagliardo--Nirenberg interpolation
inequality \cite[Part~I, Theorem~10.1]{Friedman1969}, together with the
Neumann elliptic estimate
$$
 \|f\|_{W^{2,\rho}(\Omega)}
 \le C\bigl(\|\Delta f\|_{L^\rho(\Omega)}
       +\|f\|_{L^{p_0}(\Omega)}\bigr),
$$
immediately yields the asserted inequality.
\end{proof}}



\section[Global Boundedness for the Parabolic--Elliptic--Parabolic System (kappa=0)]{Parabolic--Elliptic--Parabolic System ($\kappa=0$)}
In this section, we prove Theorem~\ref{main_theorem_kappa=0} for the
parabolic--elliptic--parabolic system, corresponding to $\kappa=0$.
In this case, system~\eqref{sys} takes the form
\begin{numcases}{}
 u_t=\nabla\cdot(D(u)\nabla u)-\chi\nabla\cdot(u\nabla v)-u-uw+\varphi,
 & $x\in\Omega,\ t>0$,\nonumber\\
 0=\Delta v-v+uw,
 & $x\in\Omega,\ t>0$,\nonumber\\
 w_t=\Delta w-w+v,
 & $x\in\Omega,\ t>0$,\nonumber\\
 \frac{\partial u}{\partial\nu}
 =\frac{\partial v}{\partial\nu}
 =\frac{\partial w}{\partial\nu}=0,
 & $x\in\partial\Omega,\ t>0$,\nonumber\\
 u(x,0)=u_0(x),
 & $x\in\Omega$,\nonumber\\
 w(x,0)=w_0(x),\qquad x\in\Omega\label{sys-kappa-0}
\end{numcases}
The proof consists of two main steps. We first derive uniform
$L^p$-bounds for $w$ below the critical Sobolev exponent. We then
establish suitable $L^p$-bounds for $u$ by exploiting the diffusion
term and the elliptic equation for $v$. Since the interpolation
argument depends essentially on the spatial dimension, the estimate
for $u$ will be treated separately in three dimensional regimes.


We begin with the following uniform $L^p$-estimate for $w$:
$$
 1\le p<\frac{n}{n-2},
$$
where $\displaystyle\frac{n}{n-2}:=\infty$ when $n=2$.

\begin{lem}\label{lem:winf}
In the parabolic--elliptic--parabolic
system \eqref{sys-kappa-0}, suppose that the assumptions \eqref{eq:basic_assumptions} and
\eqref{eq1.2} are valid.
Then for $n\ge2$ and
$$
p\in\left[1,\frac{n}{n-2}\right),
$$
there exists $C(p)>0$ such that
$$
 \sup_{t\in(0,T_{\max})}\|w(\cdot,t)\|_{L^p(\Omega)}\le C(p).
$$
\end{lem}

\begin{proof}
Integrating the first equation in \eqref{sys-kappa-0} over $\Omega$ gives
$$
 \frac{d}{dt}\int_\Omega u\,dx
 +\int_\Omega u\,dx
 +\int_\Omega uw\,dx
 =\int_\Omega\varphi\,dx.
$$
Multiplying this identity by $e^t$ and integrating over $(0,t)$,
we obtain
\begin{align*}
 \int_\Omega u(x,t)\,dx+ \int_0^t e^{s-t}\int_\Omega uw\,dx\,ds
 & =  e^{-t}\int_\Omega u_0\,dx
 +\int_0^t e^{s-t}\int_\Omega\varphi\,dx\,ds.
\end{align*}
Therefore, by \eqref{eq:basic_assumptions}, we have
\begin{equation}\label{eq:uw_weighted}
\int_0^t e^{s-t}\int_\Omega uw\,dx\,ds\le  \int_\Omega u_0\,dx+|\Omega|\varphi_*.
\end{equation}

For each fixed $s$, Lemma~\ref{lem-regularity}(ii), specifically
\eqref{eq:elliptic_L1_Lq}, applied to the equation for $v$ in
\eqref{sys-kappa-0} yields
$$
 \|v(\cdot,s)\|_{L^p}
 \le C\|u(\cdot,s)w(\cdot,s)\|_{L^1}\ \ \ \ \textrm{for}\ n\geq 2,\  \ 1\le p <\frac{n}{n-2}.
$$
Combining this estimate with \eqref{eq:uw_weighted}, we obtain for $0<t< T_{\max}$,
\begin{equation}\label{eq:v_weighted}
\int_0^t e^{s-t}\|v\|_{L^p}\,ds\le C.
\end{equation}

It remains to derive the corresponding estimate for $w$. The case
$p=1$ follows from Lemma~\ref{lem-mass}. Let therefore $p>1$. Testing the third
equation in \eqref{sys-kappa-0} by $pw^{p-1}$ and integrating by parts, we find for $0<t< T_{\max}$,
\begin{align*}
\frac{d}{dt}\int_\Omega w^p
 &=-p(p-1)\int_\Omega w^{p-2}|\nabla w|^2
 -p\int_\Omega w^p
 +p\int_\Omega vw^{p-1}\\
&\le -p\int_\Omega w^p  +p\|v\|_{L^p}\|w\|_{L^p}^{p-1}.
\end{align*}
Thus
$$
\frac{d}{dt} \|w\|_{L^p}\le  -\|w\|_{L^p}+\|v\|_{L^p}.
$$
Integrating this differential inequality and using \eqref{eq:v_weighted}, we obtain
 for $0<t< T_{\max}$,
$$
\|w\|_{L^p}\le e^{-t} \|w_0\|_{L^p}+\int_0^t e^{(s-t)} \|v\|_{L^p}\le C,
$$
provided that
$$
n\geq 2,\  \ 1< p <\frac{n}{n-2}.
$$
The proof is complete.
\end{proof}



We next derive the $L^p$-bounds for $u$ required in the proof of
Theorem~\ref{main_theorem_kappa=0}. The principal difficulty is to absorb the chemotactic
term by the diffusion term. Since the admissible interpolation
exponents depend on the spatial dimension, we divide the argument
into the following three cases.


\begin{prop}\label{prop-u-2}
In the parabolic--elliptic--parabolic
system \eqref{sys-kappa-0},
suppose that the assumptions in Lemma \ref{lem:winf} are valid,  and
$$n=2,\  \ m>1.
$$
Then, for every $p>1$, there exists
$C(p)>0$ such that
$$
 \|u(\cdot,t)\|_{L^p(\Omega)}\le C(p)
 \qquad\text{for all }t\in(0,T_{\max}).
$$
\end{prop}

\begin{prop}\label{prop-u-34}
In the parabolic--elliptic--parabolic
system \eqref{sys-kappa-0},
suppose that the assumptions in Lemma \ref{lem:winf} are valid, and
$$
n\in \{ 3,4\}, \ \ m>3-\frac4n.
$$
Then, for every
$p>1$, there exists $C(p)>0$ such that
$$
 \|u(\cdot,t)\|_{L^p(\Omega)}\le C(p)
 \qquad\text{for all }t\in(0,T_{\max}).
$$
\end{prop}

\begin{prop}\label{prop-u-n_>=4}
In the parabolic--elliptic--parabolic
system \eqref{sys-kappa-0},
suppose that the assumptions in Lemma \ref{lem:winf} are valid, and
$$
n\ge 4, \ \ m>1+\frac n4.
$$
Then, for $p$ sufficiently large, there exists
$C(p)>0$ such that
$$
 \|u(\cdot,t)\|_{L^{p}(\Omega)}\le C(p)
 \qquad\text{for all }t\in(0,T_{\max}).
$$
\end{prop}

Assuming Propositions~\ref{prop-u-2}--\ref{prop-u-n_>=4}  for the moment, we first complete
the proof of Theorem~\ref{main_theorem_kappa=0}.

%

{
\begin{proof}[Proof of Theorem~\ref{main_theorem_kappa=0}]
The preceding $L^p$-estimates, combined with the equations for $v$ and $w$,
yield the integrability required for $u\nabla v$ in the Tao--Winkler
iteration lemma. The $L^\infty$-bootstrap result of Tao--Winkler
\cite[Appendix, Lemma~4.1]{TaoWinkler2012JDE} and standard elliptic and
parabolic regularity theory yield the asserted boundedness and,
by Lemma~\ref{global_bounded_lemma}, global existence; see
\cite[Lemma~3.3]{wang2025} for this standard argument.
\end{proof}}

For the rest of this section,  we complete the proofs of Propositions \ref{prop-u-2}, \ref{prop-u-34}
 and \ref{prop-u-n_>=4} respectively.

\begin{proof}[Proof of Proposition \ref{prop-u-2}]
Fix $p>1$. Testing the first equation in \eqref{sys-kappa-0} by $pu^{p-1}$,
integrating by parts, and applying Young's inequality, we obtain
\begin{align*}
 \frac{d}{dt}\int_\Omega u^p
 &=p\int_\Omega u^{p-1}\nabla\!\cdot(D(u)\nabla u)
   -\chi p\int_\Omega u^{p-1}\nabla\!\cdot(u\nabla v)\\
 &\quad-p\int_\Omega u^p-p\int_\Omega u^pw
   +p\int_\Omega u^{p-1}\varphi\\
 &\le -p(p-1)\int_\Omega u^{p-2}D(u)|\nabla u|^2
   -\chi(p-1)\int_\Omega u^p\Delta v\\
 &\quad-\int_\Omega u^p+\int_\Omega\varphi^p.
\end{align*}
Using the assumptions on $D$ and $\varphi$, we therefore find
\begin{align}
 \frac{d}{dt}\int_\Omega u^p+\int_\Omega u^p
 & \le  - k_D p(p-1)\int_\Omega u^{p+m-3}|\nabla u|^2\notag\\
 &\quad
   +\chi(p-1)\int_\Omega u^p | \Delta v |
   +\int_\Omega\varphi^p \notag\\
& \le -\frac{4k_Dp(p-1)}{(p+m-1)^2}
 \int_\Omega\left|\nabla u^{\frac{p+m-1}{2}}\right|^2\notag\\
 &\quad+\chi(p-1)\int_\Omega u^p|\Delta v|+|\Omega|\varphi_*^p.
 \label{eq:energybase_n2}
\end{align}

Thus, it remains to estimate the chemotactic term in \eqref{eq:energybase_n2}
and show that it can be absorbed by the negative diffusion term.

We next estimate the chemotactic term in \eqref{eq:energybase_n2}
$$
\int_\Omega u^p|\Delta v|.
$$
Let $p_0>1$ be specified below. By Lemma~\ref{lem:winf},
 when $n=2$,
$$
 \sup_{t\in(0,T_{\max})}\|w(\cdot,t)\|_{L^{p_0}}\le C.
$$
Choose conjugate exponents $q,q'>1$ such that
$$
 \frac1q+\frac1{q'}=1,
 \qquad
 \frac1q=\frac1{p q'}+\frac1{p_0}.
$$
Solving these relations gives
\begin{equation}\label{eq:qprime_2d}
 q'=\left(1+\frac1p\right)
 \left(1-\frac1{p_0}\right)^{-1}>1,
 \qquad
 q=\left(1-\frac1{q'}\right)^{-1}>1.
\end{equation}
Then H\"older's inequality and elliptic regularity for $-\Delta v+v=uw$  yield
\begin{align}
 \int_\Omega u^p|\Delta v|
 &\le \|u\|_{L^{p q'}}^p\|\Delta v\|_{L^q}
 &\le C\|u\|_{L^{p q'}}^p\|uw\|_{L^q}\notag\\
 &\le C\|u\|_{L^{p q'}}^{p+1}\|w\|_{L^{p_0}}
 \le C\|u\|_{L^{p q'}}^{p+1}.
 \label{eq:cross_2d}
\end{align}

{By Lemma~\ref{lem:GN}(i) and the mass estimate in
Lemma~\ref{lem-mass},}
\begin{align}
  \left\|u^{\frac{p+m-1}{2}}\right\|_{L^{\frac{2pq'}{p+m-1}}}
  & \le  C\left\|\nabla u^{\frac{p+m-1}{2}}\right\|_{L^2}^{a}
  \left\|u^{\frac{p+m-1}{2}}\right\|_{L^{\frac{2}{p+m-1}}}^{1-a}\notag\\
  &\quad+C\left\|u^{\frac{p+m-1}{2}}\right\|_{L^{\frac{2}{p+m-1}}}\notag\\
  & \le  C\left\|\nabla u^{\frac{p+m-1}{2}}\right\|_{L^2}^{a} + C.
 \label{eq:GN_2d}
\end{align}
where, by \eqref{eq:qprime_2d},
\begin{equation}\label{eq:a_2d}
 a=1-\frac1{p q'}
 =1-\frac1{p+1}\left(1-\frac1{p_0}\right)\in(0,1).
\end{equation}
Combining \eqref{eq:energybase_n2}, \eqref{eq:cross_2d}, and
\eqref{eq:GN_2d}, we obtain
\begin{align}
\frac{d}{dt}\int_\Omega u^p+\int_\Omega u^p
 &\le -\frac{4k_Dp(p-1)}{(p+m-1)^2}
 \int_\Omega\left|\nabla u^{\frac{p+m-1}{2}}\right|^2\notag\\
 &\quad+C\|u\|_{L^{pq'}}^{p+1}+|\Omega|\varphi_*^p\notag\\
 &= -\frac{4k_Dp(p-1)}{(p+m-1)^2}
 \left\|\nabla u^{\frac{p+m-1}{2}}\right\|_{L^2}^2\notag\\
 &\quad+C\left\|u^{\frac{p+m-1}{2}}\right\|_{L^{\frac{2pq'}{p+m-1}}}^{\frac{2(p+1)}{p+m-1}}
 +|\Omega|\varphi_*^p\notag\\
 &\le -\frac{4k_Dp(p-1)}{(p+m-1)^2}
 \left\|\nabla u^{\frac{p+m-1}{2}}\right\|_{L^2}^2\notag\\
 &\quad+C\left\|\nabla u^{\frac{p+m-1}{2}}\right\|_{L^2}^{\frac{2a(p+1)}{p+m-1}}+C.
 \label{eq:final_2d}
\end{align}
By Young's inequality, the positive gradient term in \eqref{eq:final_2d}
can be absorbed whenever
$$
 \frac{2a(p+1)}{p+m-1}<2.
$$
In view of \eqref{eq:a_2d}, this condition is equivalent to
$$
 m>1+\frac1{p_0}.
$$
Since $m>1$, we may choose $p_0$ sufficiently large so that this
inequality holds. Thus
\eqref{eq:final_2d} gives
$$
 \frac{d}{dt}\int_\Omega u^p+\int_\Omega u^p\le C,
$$
and the asserted bound follows by comparison.
\end{proof}

We next prove Proposition \ref{prop-u-34}, which concerns the dimensions
$n\in \{3,4\}$.
The proof proceeds along the same lines as that of
Proposition \ref{prop-u-2}. We therefore omit the calculations that are
independent of the space dimension and focus on the steps where the
dimension plays an essential role.

\begin{proof}[Proof of Proposition \ref{prop-u-34}]
We deal with the dimension $n\in \{3,4\}$.
Fix $p>1$. Arguing as in the derivation of \eqref{eq:energybase_n2}, we obtain
\begin{equation}\label{0-lm-u-pf-key}
 \frac{d}{dt}\int_\Omega u^p+\int_\Omega u^p
 \le -\frac{4k_Dp(p-1)}{(p+m-1)^2}
 \int_\Omega\left|\nabla u^{\frac{p+m-1}{2}}\right|^2
 +\chi(p-1)\int_\Omega u^p|\Delta v|+|\Omega|\varphi_*^p.
\end{equation}
Choose
$$
1<p_0<\frac{n}{n-2},
$$
to be specified below. By Lemma~\ref{lem:winf},
$$
 \sup_{t\in(0,T_{\max})}\|w(\cdot,t)\|_{L^{p_0}}\le C(p_0).
$$
Proceeding as in \eqref{eq:cross_2d}, we obtain
\begin{align} \label{eq:cross_low_dim}
 \int_\Omega u^p|\Delta v|
 \le \|u\|_{L^{p q'}}^p\|\Delta v\|_{L^q}
 \le C\|u\|_{L^{p q'}}^p\|uw\|_{L^q}
 \le C\|u\|_{L^{p q'}}^{p+1},
\end{align}
provided that
$$
 q,q'>1,\qquad
 \frac1q+\frac1{q'}=1,
 \qquad
 \frac1q=\frac1{p q'}+\frac1{p_0}.
$$
Solving these relations gives
\begin{equation}\label{0-lm-u-pf-qq'}
 q'=\left(1+\frac1p\right)
 \left(1-\frac1{p_0}\right)^{-1}>1,
 \qquad
 q=\left(1-\frac1{q'}\right)^{-1}>1
\end{equation}
for the fixed $p>1$ and $1<p_0<n/(n-2)$.

{Lemma~\ref{lem:GN}(i) and the mass estimate in
Lemma~\ref{lem-mass} yield}
\begin{align}
 \left\|u^{\frac{p+m-1}{2}}\right\|_{L^{\frac{2p q'}{p+m-1}}}
 &\le C\left\|\nabla u^{\frac{p+m-1}{2}}\right\|_{L^2}^{a}
 \left\|u^{\frac{p+m-1}{2}}\right\|_{L^{\frac2{p+m-1}}}^{1-a} +C\left\|u^{\frac{p+m-1}{2}}\right\|_{L^{\frac2{p+m-1}}}\notag\\
 &\le C\left\|\nabla u^{\frac{p+m-1}{2}}\right\|_{L^2}^{a}+C,
 \label{eq:GN_standard_low}
\end{align}
where $a$ is determined by
$$
 \frac{p+m-1}{2p q'}
 =a\left(\frac12-\frac1n\right)
 +(1-a)\frac{p+m-1}{2}.
$$
In view of \eqref{0-lm-u-pf-qq'}, this identity can be written as
\begin{equation}\label{0-lm-u-pf-a1}
 \frac{p+m-1}{2}\frac1{p+1}\left(1-\frac1{p_0}\right)
 =a\left(\frac12-\frac1n\right)
 +(1-a)\frac{p+m-1}{2}.
\end{equation}

Combining \eqref{0-lm-u-pf-key}, \eqref{eq:cross_low_dim}, and
\eqref{eq:GN_standard_low}, we find
\begin{align}\label{0-lm-u-pf-last}
 \frac{d}{dt}\int_\Omega u^p+\int_\Omega u^p
 &\le -\frac{4k_Dp(p-1)}{(p+m-1)^2}
 \left\|\nabla u^{\frac{p+m-1}{2}}\right\|_{L^2}^2 +C\left\|\nabla u^{\frac{p+m-1}{2}}\right\|_{L^2}^{\frac{2a(p+1)}{p+m-1}}+C.
\end{align}
The positive gradient term can be absorbed by Young's inequality if
\begin{equation}\label{0-lm-u-pf-a2}
 \frac{2a(p+1)}{p+m-1}<2.
\end{equation}

We next show that $p_0$ can be chosen so that
$$
 p_0\in\left(1,\frac{n}{n-2}\right),\quad a\in(0,1),
$$
and \eqref{0-lm-u-pf-a2} holds.

First, by \eqref{0-lm-u-pf-a1}, condition \eqref{0-lm-u-pf-a2} is equivalent to
$$
 m>2-\frac{2}{n}+\frac{1}{p_0}.
$$
Since $m>3-\frac4n$, this inequality holds whenever $p_0$ is chosen
sufficiently close to $\frac{n}{n-2}$.

We then verify that the same choice ensures $a\in(0,1)$.
{By \eqref{0-lm-u-pf-a1}, this is equivalent to
\begin{equation}\label{pf-prop-0-34-a}
 \frac12-\frac1n
 <\frac{p+m-1}{2}\frac1{p+1}\left(1-\frac1{p_0}\right)
 <\frac{p+m-1}{2},
\end{equation}}
The second
inequality in \eqref{pf-prop-0-34-a}   holds for every $p>1$ and $p_0>1$.
For the first one, we only consider the limiting case that
$$
p_0 = \frac{n}{n-2},\ \ m>3-\frac{4}{n}.
$$
Then it is direct to verify that  for
$$
n\in\{3,4\},\ \  m>3-\frac{4}{n},
$$
we have
\begin{align*}
 \frac{p+m-1}{2}\frac1{p+1}\left(1-\frac1{p_0}\right)+ \frac1n
=  \frac{p+m-1}{2}\frac1{p+1}\frac{2}{n}+ \frac1n = \frac1n  \frac{2p+m}{p+1}> \frac{1}{2}.
\end{align*}
By continuity, the inequality
remains valid for $p_0<\frac{n}{n-2}$ sufficiently close to this
endpoint. Hence $a\in(0,1)$, and the required choice of $p_0$ is
possible. It is worth pointing out that
$$
\lim_{p\rightarrow +\infty} \frac1n  \frac{2p+m}{p+1} = \frac{2}{n}.
$$
This computation also explains why the direct interpolation argument
cannot cover $n\ge5$ for arbitrarily large $p$.

We have therefore shown that $p_0$ can be chosen sufficiently close
to $\frac{n}{n-2}$ so that $a\in(0,1)$ and \eqref{0-lm-u-pf-a2} hold
simultaneously. Applying Young's inequality to \eqref{0-lm-u-pf-last}, we obtain
$$
 \frac{d}{dt}\int_\Omega u^p+\int_\Omega u^p\le C.
$$
The asserted uniform $L^p$-bound now follows by comparison. This
completes the proof of Proposition~\ref{prop-u-34}.
\end{proof}


The preceding argument relies on keeping the target exponent in the
Gagliardo--Nirenberg inequality below the critical Sobolev exponent.
As indicated by the limiting relation above, this requirement cannot
be met for arbitrarily large $p$ when $n\ge5$. We therefore adopt a
different strategy in the proof of Proposition~\ref{prop-u-n_>=4}: instead of
estimating the chemotactic term directly, we retain its product
structure and perform two successive interpolation steps for $w$.
Notably, when $n=4$, the two approaches overlap and yield the same
threshold $m>2$.

\begin{proof}[Proof of Proposition \ref{prop-u-n_>=4}]
Let $n\ge4$ and $m>1+\frac n4$. Since $n\ge4$, we have
$$
 1<\frac{n}{n-2}.
$$
Moreover, the assumption on $m$ is equivalent to
$$
 \frac{n}{4m-6}<\frac{n}{n-2}.
$$
Consequently,
$$
 \max\left\{1,\frac{n}{4m-6}\right\}
 <\frac{n}{n-2},
$$
so the interval below is nonempty. We fix
$$
 p_0\in
 \left(
 \max\left\{1,\frac{n}{4m-6}\right\},
 \frac{n}{n-2}
 \right).
$$
The upper bound on $p_0$ allows us to apply
Lemma~\ref{lem:winf}, which yields
$$
 \sup_{t\in(0,T_{\max})}
 \|w(\cdot,t)\|_{L^{p_0}(\Omega)}
 \le C(p_0).
$$
The condition $p_0>1$ will allow us to interpolate between
$L^{p_0}$ and $L^\gamma$ for some  $\gamma>p_0$, whereas
$p_0>\frac{n}{4m-6}$ will be used  in the final verification of the absorption conditions.

Arguing as in the proof of Proposition \ref{prop-u-2}, we obtain
\begin{equation}\label{lm-u-pf-key-higher}
 \frac{d}{dt}\int_\Omega u^p+\int_\Omega u^p
 \le -\frac{4k_Dp(p-1)}{(p+m-1)^2}
 \int_\Omega\left|\nabla u^{\frac{p+m-1}{2}}\right|^2
 +\chi(p-1)\int_\Omega u^p|\Delta v|+|\Omega|\varphi_*^p.
\end{equation}
We first derive preliminary estimates for the diffusion and
chemotactic terms in \eqref{lm-u-pf-key-higher} as follows
$$
\int_\Omega\left|\nabla u^{\frac{p+m-1}{2}}\right|^2\ \ \textrm{and}\ \ \int_\Omega u^p|\Delta v|.
$$
\begin{itemize}
\item {According to Lemma~\ref{lem:GN}(i) and the mass
estimate in Lemma~\ref{lem-mass}, one has}
\begin{align*}
 \|u^{\frac{p+m-1}{2}}\|_{L^{\frac{2n}{n-2}}}
 &\le C\|\nabla u^{\frac{p+m-1}{2}}\|_{L^2}
    +C\|u^{\frac{p+m-1}{2}}\|_{L^{\frac2{p+m-1}}}\\
 &\le C\|\nabla u^{\frac{p+m-1}{2}}\|_{L^2}+C,
\end{align*}
i.e.,
$$
C\int_\Omega\left|\nabla u^{\frac{p+m-1}{2}}\right|^2 +C\geq  \left\|u^{\frac{p+m-1}{2}}\right\|_{L^{\frac{2n}{n-2}}}^2 =\left\|u\right\|_{L^{\frac{n(p+m-1)}{n-2}}}^{p+m-1}.
$$
\item Again, by H\"older's inequality and elliptic regularity for $-\Delta v+v=uw$, we obtain
\begin{align}\label{lm-pf-higher-mixed}
 \int_\Omega u^p|\Delta v|
 \le \|u\|_{L^{p \sigma}}^p\|\Delta v\|_{L^{\sigma'}}
 \le C\|u\|_{L^{p \sigma}}^p\|uw\|_{L^{\sigma'}}
 \le C\|u\|_{L^{p \sigma}}^{p+1}\|w\|_{L^{\beta}},
\end{align}
provided that
$$
 \sigma,\sigma',\beta>1,\qquad
 \frac1\sigma+\frac1{\sigma'}=1,
 \qquad
 \frac1{\sigma'}=\frac1{p \sigma}+\frac1{\beta}.
$$
\end{itemize}
Thus, thanks to the above two preliminary estimates, \eqref{lm-u-pf-key-higher} becomes
\begin{equation}\label{lm-u-pf-key-higher-1}
 \frac{d}{dt}\int_\Omega u^p+\int_\Omega u^p
 \le - C_0 \left\|u\right\|_{L^{\frac{n(p+m-1)}{n-2}}}^{p+m-1}
 +C\|u\|_{L^{p \sigma}}^{p+1}\|w\|_{L^{\beta}}+C.
\end{equation}

In particular, we choose $\sigma$ such that
\begin{equation}\label{lm-pf-sigma-def}
p\sigma =\frac{n(p+m-1)}{n-2}  \ \  \textrm{i.e.}, \ \  \frac{1}{\sigma} = \frac{p}{p+m-1}\frac{n-2}{n}.
\end{equation}
and thus
\begin{equation}\label{lm-pf-higher-beta-sigma'}
\frac{1}{\sigma'}= 1-\frac{1}{\sigma}=1- \frac{p}{p+m-1}\frac{n-2}{n},\ \
\frac{1}{\beta}=1- \frac{p+1}{p \sigma} =1- \frac{p+1}{p+m-1}\frac{n-2}{n}.
\end{equation}
Obviously, if $n\geq 4$ and $\displaystyle m>1+\frac{n}{4}$, for $p>1$,
\begin{equation}\label{lm-pf-higher-sigma-sigma'}
\frac{1}{\sigma}<\frac{n-2}{n},\ \ \frac{1}{\sigma'}>\frac{2}{n}.
\end{equation}
Thanks to  the choice of $\sigma$ designated in \eqref{lm-pf-sigma-def}, the inequality \eqref{lm-u-pf-key-higher-1} becomes
\begin{align*}
\frac{d}{dt}\int_\Omega u^p+\int_\Omega u^p
 & \le -C_0 \|u\|_{L^{p\sigma}}^{p+m-1}
 +C\|u\|_{L^{p\sigma}}^{p+1}\|w\|_{L^{\beta}} + C.
\end{align*}
This leads to
\begin{equation}\label{lm-u-pf-key-higher-2}
\int_\Omega u^p \le -C_0\int_0^t e^{s-t} \|u\|_{L^{p\sigma}}^{p+m-1}\,ds+  C\int_0^t e^{s-t} \|u\|_{L^{p\sigma}}^{p+1}\|w\|_{L^\beta}\,ds+C.
\end{equation}

The remaining task is to estimate
$$
\int_0^t e^{s-t} \|u\|_{L^{p\sigma}}^{p+1}\|w\|_{L^\beta}\,ds,
$$
in terms of the first, negative integral on the right-hand side of
\eqref{lm-u-pf-key-higher-2}.

Let $p_0<\beta<\gamma$. Interpolation between $L^{p_0}$ and
$L^\gamma$ gives  the existence of $\theta \in (0,1)$  such that
\begin{equation}\label{lm-pf-higher-simple}
         \|w\|_{L^\beta} \le \|w\|_{L^{p_0}}^{1-\theta}\|w\|_{L^\gamma}^{\theta}\le C\|w\|_{L^\gamma}^{\theta},
\end{equation}
where
\begin{equation}\label{lm-pf-theta}
 \frac1\beta=\frac{1-\theta}{p_0}+\frac\theta\gamma.
\end{equation}
By the definition of $\beta$ in \eqref{lm-pf-higher-beta-sigma'},  for
$\displaystyle n\geq 4,\ \ m>1+\frac{n}{4},$
$$
\frac{1}{\beta}=1- \frac{p+1}{p+m-1}\frac{n-2}{n}> 1-  \frac{n-2}{n} =\frac{2}{n}.
$$
Thus  we fix
$$
\gamma>  \frac{n}{2} > \beta.
$$
Moreover,  for $n \geq 4$, according to the choice of $p_0$ at the beginning of this proof,
$$
\lim_{p\rightarrow +\infty } \frac{1}{\beta}= \lim_{p\rightarrow +\infty } 1- \frac{p+1}{p+m-1}\frac{n-2}{n} =  \frac{2}{n}\le  \frac{n-2}{n}<\frac{1}{p_0},
$$
which indicates that  there exists $P_1>1$, which depends on $n, m, p_0$ only, such that $\beta>p_0$ for $p> P_1$.
Hence,
\begin{equation*}
p_0<\beta<\gamma\ \ \textrm{for}\  p> P_1,
\end{equation*}
and  the existence of $\theta\in (0,1)$ satisfying  \eqref{lm-pf-theta} is guaranteed.

By \eqref{lm-pf-higher-simple} and H\"older's inequality,
\begin{align}
 \int_0^t e^{s-t} \|u\|_{L^{p\sigma}}^{p+1}\|w\|_{L^\beta}\,ds
 &\le C\int_0^t e^{s-t} \|u\|_{L^{p\sigma}}^{p+1}
 \|w\|_{L^\gamma}^{\theta}\,ds\notag\\
 &\le C \left(\int_0^t e^{s-t}
 \|u\|_{L^{p\sigma}}^{p+m-1}\,ds\right)^{1-\frac{1}{k}}\notag\\
 &\quad\times
 \left(\int_0^t e^{s-t}\|w\|_{L^\gamma}^{\theta k}\,ds\right)^{\frac{1}{k}}.
 \label{lm-pf-higher-split}
\end{align}
where
$$
 k:=\frac{p+m-1}{m-2},
$$
and when $n\geq 4$, $\displaystyle m>1+\frac{n}{4}$,  $p>P_1$,  the conjugate exponents
{$$
\frac{k}{k-1}= \frac{p+m-1}{p+1}
$$}
exceed one.

Secondly, we focus on the following term in \eqref{lm-pf-higher-split}
$$
\int_0^t e^{s-t}\|w\|_{L^\gamma}^{\theta k}\,ds.
$$
{By Lemma~\ref{lem:GN}(ii), if}
$$
\gamma >p_0,\ \ \frac1\rho-\frac2n<0
$$
there exists $\alpha\in (0,1)$ such  that
\begin{align*}
 \|w\|_{L^\gamma}
 \le C\|\Delta w\|_{L^\rho}^{\alpha}
       \|w\|_{L^{p_0}}^{1-\alpha}
       +C\|w\|_{L^{p_0}} \le C\|\Delta w\|_{L^\rho}^{\alpha}+C,
\end{align*}
where
\begin{equation}\label{lm-pf-alpha}
\frac1\gamma
 =\alpha\left(\frac1\rho-\frac2n\right)+\frac{1-\alpha}{p_0}.
\end{equation}
In view of \eqref{lm-pf-higher-sigma-sigma'}, we may define $\rho>1$ by
\begin{equation}\label{lm-pf-rho}
\frac{1}{\rho} =\frac{1}{\sigma'}- \frac{2}{n}>0.
\end{equation}
Using the definitions of $\sigma'$
in \eqref{lm-pf-higher-beta-sigma'} and $\rho$ in \eqref{lm-pf-rho}, we further compute
$$
\frac1\rho-\frac2n =  \frac{1}{\sigma'}- \frac{4}{n}<0
$$
provided that
$$
p>\frac{(n-4)(m-1)}{2}.
$$
Therefore, when
$$
p>\max \left\{ P_1, \ \frac{(n-4)(m-1)}{2} \right\},
$$
then $\gamma> \beta>p_0$, and thus there exists unique  $\alpha\in (0,1)$ satisfying \eqref{lm-pf-alpha}.

Then it follows from  Lemma \ref{max_sobolev_regularity} and the equation satisfied by $w$ that whenever
\begin{equation}\label{atk>1}
\alpha\theta k=\alpha\theta \cdot\frac{p+m-1}{m-2}>1,
\end{equation}
we  have
\begin{align}\label{lm-pf-higher-w-1}
\int_0^t e^{s-t}\|w\|_{L^\gamma}^{\theta k}\,ds
 \le C\int_0^t e^{s-t}\|\Delta w\|_{L^\rho}^{\alpha\theta k}\,ds+C \le C\int_0^t e^{s-t}\|v\|_{L^\rho}^{\alpha\theta k}\,ds+C.
\end{align}
Moreover, according to the choice of $\rho$ in \eqref{lm-pf-rho},  we have the Sobolev embedding
$$
W^{2,\sigma'}(\Omega)\hookrightarrow L^\rho(\Omega).
$$
Thus, similar to \eqref{lm-pf-higher-mixed},   we obtain
$$
\|v\|_{L^\rho} \le C\|v\|_{W^{2,\sigma'}}\le C \|uw\|_{L^{\sigma'}}\le C\|u\|_{L^{p \sigma}}\|w\|_{L^{\beta}},
$$
which,  together with \eqref{lm-pf-higher-simple}, \eqref{lm-pf-higher-w-1} and Young's inequality, yields
\begin{align*}
\int_0^t e^{s-t}\|w\|_{L^\gamma}^{\theta k}\,ds
& \le C\int_0^t e^{s-t}\|u\|_{L^{p \sigma}}^{\alpha\theta k}\|w\|_{L^{\beta}}^{\alpha\theta k}\,ds+C\\
& \le C\int_0^t e^{s-t}\|u\|_{L^{p \sigma}}^{\alpha\theta k}\|w\|_{L^{\gamma}}^{\alpha\theta^2 k}\,ds+C\\
& \le \frac{1}{2}\int_0^t e^{s-t}\|w\|_{L^\gamma}^{\theta k}\,ds+ C\int_0^t e^{s-t}\|u\|_{L^{p \sigma}}^{\frac{\alpha\theta}{1-\alpha\theta} k}\,ds+C.
\end{align*}
Absorbing the first term on the right-hand side into the
left-hand side, we obtain
\begin{equation}\label{lm-pf-higher-w-2}
\int_0^t e^{s-t}\|w\|_{L^\gamma}^{\theta k}\,ds
\le  C\int_0^t e^{s-t}\|u\|_{L^{p \sigma}}^{\frac{\alpha\theta}{1-\alpha\theta} k}\,ds+C.
\end{equation}

Substituting \eqref{lm-pf-higher-w-2} into \eqref{lm-pf-higher-split}, we obtain
\begin{align*}
 & \int_0^t e^{s-t} \|u\|_{L^{p\sigma}}^{p+1}\|w\|_{L^\beta}\,ds \\
 & \le C \left(\int_0^t e^{s-t}
 \|u\|_{L^{p\sigma}}^{p+m-1}\,ds\right)^{1-\frac{1}{k}}\\
 &\quad\times
 \left(\int_0^t e^{s-t}\|u\|_{L^{p \sigma}}^{\frac{\alpha\theta}{1-\alpha\theta} k}\,ds+1\right)^{\frac{1}{k}}\\
 &\le C \left(\int_0^t e^{s-t}
 \|u\|_{L^{p\sigma}}^{p+m-1}\,ds\right)^{1-\frac{1}{k}}\\
 &\quad\times
 \left(\int_0^t e^{s-t}\|u\|_{L^{p \sigma}}^{\frac{\alpha\theta}{1-\alpha\theta} k}\,ds\right)^{\frac{1}{k}}\\
 &\quad+C \left(\int_0^t e^{s-t}
 \|u\|_{L^{p\sigma}}^{p+m-1}\,ds\right)^{1-\frac{1}{k}}.
\end{align*}
Thanks to this crucial estimate, it follows from \eqref{lm-u-pf-key-higher-2} that
\begin{align}
\int_\Omega u^p \le &-C_0\int_0^t e^{s-t} \|u\|_{L^{p\sigma}}^{p+m-1}\,ds \notag\\
&+ C \left(\int_0^t e^{s-t}
 \|u\|_{L^{p\sigma}}^{p+m-1}\,ds\right)^{1-\frac{1}{k}}
 \left(\int_0^t e^{s-t}\|u\|_{L^{p \sigma}}^{\frac{\alpha\theta}{1-\alpha\theta} k}\,ds\right)^{\frac{1}{k}}\notag\\
 &+C \left(\int_0^t e^{s-t}
 \|u\|_{L^{p\sigma}}^{p+m-1}\,ds\right)^{1-\frac{1}{k}}+C.
 \label{lm-u-pf-newkey-higher-3}
\end{align}

Before  we complete the last piece of the proof, we claim that {\it if
$$
n\geq 4,\ \ m>1+\frac{n}{4},\ \ p_0 \in \left( \frac{n}{4m-6}, \frac{n}{n-2}\right),
$$
then there exists
$$
P_2>\max \left\{ P_1, \ \frac{(n-4)(m-1)}{2} \right\},
$$
such that for $p>P_2$, \eqref{atk>1} and
\begin{equation}\label{prop-pf-alpha-theta}
\frac{\alpha\theta}{1-\alpha\theta} < m-2,\  \  \textrm{i.e.},\ \ \alpha\theta<1-\frac{1}{m-1}
\end{equation}
hold simultaneously}

Suppose the claim is valid. Recall
$$
 k:=\frac{p+m-1}{m-2}.
$$
Thus
$$
\frac{\alpha\theta}{1-\alpha\theta} k <p+m-1.
$$
By \eqref{lm-u-pf-newkey-higher-3} and Young's inequality, we have
\begin{align*}
\int_\Omega u^p \le &-C_0\int_0^t e^{s-t}
 \|u\|_{L^{p\sigma}}^{p+m-1}\,ds
 + \frac{C_0}{4} \int_0^t e^{s-t}
 \|u\|_{L^{p\sigma}}^{p+m-1}\,ds\\
 &+C \int_0^t e^{s-t}
 \|u\|_{L^{p \sigma}}^{\frac{\alpha\theta}{1-\alpha\theta} k}\,ds
 + \frac{C_0}{4} \int_0^t e^{s-t}
 \|u\|_{L^{p\sigma}}^{p+m-1}\,ds+C \\
\le &-C_0\int_0^t e^{s-t}
 \|u\|_{L^{p\sigma}}^{p+m-1}\,ds
 + \frac{C_0}{4} \int_0^t e^{s-t}
 \|u\|_{L^{p\sigma}}^{p+m-1}\,ds \\
 &+\frac{C_0}{4} \int_0^t e^{s-t}
 \|u\|_{L^{p\sigma}}^{p+m-1}\,ds
 + \frac{C_0}{4} \int_0^t e^{s-t}
 \|u\|_{L^{p\sigma}}^{p+m-1}\,ds+C\\
 \le & \ C.
\end{align*}
The desired conclusion follows.


It remains to prove the claim.
According to \eqref{lm-pf-theta} and \eqref{lm-pf-alpha},
$$
\alpha\theta = \frac{\frac{1}{p_0} -\frac{1}{\beta}}{\frac{1}{p_0}-\frac{1}{\rho}+\frac{2}{n}}= \frac{\frac{1}{p_0}-1+\frac{p+1}{p+m-1}\frac{n-2}{n}}{\frac{1}{p_0}-1+\frac{p}{p+m-1}\frac{n-2}{n}+\frac{4}{n}},
$$
and obviously,
$$
\lim_{p\rightarrow\infty}\alpha \theta= \frac{\frac{1}{p_0}-\frac{2}{n}}{\frac{1}{p_0}+\frac{2}{n}},
$$
Since $p_0<\frac{n}{n-2}$, the above limit is positive. Moreover, as $p\to\infty$
$$
k=\frac{p+m-1}{m-2} \to\infty.
$$
Hence
$\alpha\theta k>1$ for all sufficiently large $p$.

On the other hand,
$$
 \frac{\frac1{p_0}-\frac2n}
      {\frac1{p_0}+\frac2n}
 <1-\frac1{m-1}
$$
is equivalent to
$$
 p_0>\frac{n}{4m-6},
$$
which holds by the choice of $p_0$. By continuity, \eqref{prop-pf-alpha-theta}
therefore also holds for all sufficiently large $p$. We may thus
choose $P_2$ so that \eqref{atk>1} and \eqref{prop-pf-alpha-theta} hold simultaneously
whenever $p>P_2$.
The claim is proved.
\end{proof}

\section[Global Boundedness for the Fully Parabolic System (kappa=1)]{Fully Parabolic System ($\kappa=1$)}

In this section, we prove Theorem~\ref{main_theorem_kappa=1} for the
fully parabolic case. When $\kappa=1$, system~\eqref{sys} reads
\begin{numcases}{}
 u_t=\nabla\cdot(D(u)\nabla u)-\chi\nabla\cdot(u\nabla v)-u-uw+\varphi,
 & $x\in\Omega,\ t>0$,\nonumber\\
 v_t=\Delta v-v+uw,
 & $x\in\Omega,\ t>0$,\nonumber\\
 w_t=\Delta w-w+v,
 & $x\in\Omega,\ t>0$,\nonumber\\
 \frac{\partial u}{\partial\nu}
 =\frac{\partial v}{\partial\nu}
 =\frac{\partial w}{\partial\nu}=0,
 & $x\in\partial\Omega,\ t>0$,\nonumber\\
 u(x,0)=u_0(x),\quad v(x,0)=v_0(x),
 & $x\in\Omega$,\nonumber\\
 w(x,0)=w_0(x),\qquad x\in\Omega\label{sys-kappa-1}
\end{numcases}

The proof follows the same overall strategy as in Section~3. We first
derive uniform $L^p$-bounds for $w$ below the critical Sobolev exponent
and then establish uniform $L^p$-bounds for $u$ for all sufficiently
large $p$.

The principal difference from the parabolic--elliptic--parabolic case
is that $v$ now satisfies a parabolic equation. Consequently, $\Delta v$
can no longer be estimated pointwise in time by $uw$. To overcome this
difficulty, we combine weighted maximal Sobolev regularity with an
argument that preserves the product structure of the chemotactic term.

\begin{lem}\label{lp_w_kappa1}
In the fully parabolic
system~\eqref{sys-kappa-1}, suppose that the assumptions \eqref{eq:basic_assumptions} and
\eqref{eq1.3} are valid.
Then for $n\ge2$ and
$$
p\in\left[1,\frac{n}{n-2}\right),
$$
there exists $C(p)>0$ such that
$$
 \sup_{t\in(0,T_{\max})}\|w(\cdot,t)\|_{L^p(\Omega)}\le C(p).
$$
\end{lem}

\begin{proof}
The source term in the $w$-equation is $v$. By Lemma~\ref{lem-mass},
$$
 \sup_{t\in(0,T_{\max})}\|v(\cdot,t)\|_{L^1(\Omega)}\le C.
$$
Applying Lemma~\ref{lem-regularity}(i), specifically
\eqref{eq:parabolic_L1_Lq}, to the equation
$$
 w_t=\Delta w-w+v
$$
with $f=v$, we obtain the asserted estimate for every
$1\le p<\frac{n}{n-2}$.
\end{proof}

We next establish uniform $L^p$-bounds for $u$ for all sufficiently
large $p$. The bounds for the remaining finite exponents then follow
from H\"older's inequality and the boundedness of $\Omega$.

\begin{prop}
\label{prop-parabolic}
In the fully parabolic system~\eqref{sys-kappa-1},
suppose that the assumptions in Lemma \ref{lp_w_kappa1} are valid, and
$$
n\ge 4, \ \ m>1+\frac n4.
$$
Then, for $p$ sufficiently large, there exists
$C(p)>0$ such that
$$
 \|u(\cdot,t)\|_{L^{p}(\Omega)}\le C(p)
 \qquad\text{for all }t\in(0,T_{\max}).
$$
\end{prop}

Assuming Proposition~\ref{prop-parabolic} for the moment, we first
complete the proof of Theorem~\ref{main_theorem_kappa=1}.


{
\begin{proof}[Proof of Theorem~\ref{main_theorem_kappa=1}]
The preceding $L^p$-estimates, combined with the equations for $v$ and $w$,
yield the integrability required for $u\nabla v$ in the Tao--Winkler
iteration lemma. The $L^\infty$-bootstrap result of Tao--Winkler
\cite[Appendix, Lemma~4.1]{TaoWinkler2012JDE} and standard regularity theory
yield the asserted boundedness and, by
Lemma~\ref{global_bounded_lemma}, global existence; see
\cite[Lemma~4.3]{wang2025} for this standard argument.
\end{proof}}

It remains to prove Proposition~\ref{prop-parabolic}. Its proof
parallels that of Proposition~\ref{prop-u-n_>=4}; we therefore omit
the steps that remain unchanged and focus on the modifications required
by the parabolic equation for $v$.
\begin{proof}[Proof of Proposition \ref{prop-parabolic}]
Let $n\ge4$ and $m>1+\frac n4$. We choose $p_0$ exactly as in the
proof of Proposition~\ref{prop-u-n_>=4}, namely
$$
 p_0\in
 \left(
 \max\left\{1,\frac{n}{4m-6}\right\},
 \frac{n}{n-2}
 \right).
$$
By Lemma~\ref{lp_w_kappa1},
$$
 \sup_{t\in(0,T_{\max})}
 \|w(\cdot,t)\|_{L^{p_0}(\Omega)}\le C(p_0).
$$

Fix $p>1$. Arguing as in the derivation of
\eqref{lm-u-pf-key-higher}, we obtain
$$
 \frac{d}{dt}\int_\Omega u^p+\int_\Omega u^p
 \le -\frac{4k_Dp(p-1)}{(p+m-1)^2}
 \int_\Omega\left|\nabla u^{\frac{p+m-1}{2}}\right|^2
 +\chi(p-1)\int_\Omega u^p|\Delta v|+|\Omega|\varphi_*^p,
$$
{while Lemma~\ref{lem:GN}(i) and the mass estimate in
Lemma~\ref{lem-mass} yield}
$$
C\int_\Omega\left|\nabla u^{\frac{p+m-1}{2}}\right|^2 +C\geq  \left\|u^{\frac{p+m-1}{2}}\right\|_{L^{\frac{2n}{n-2}}}^2 =\left\|u\right\|_{L^{\frac{n(p+m-1)}{n-2}}}^{p+m-1}.
$$
Consequently,
$$
 \frac{d}{dt}\int_\Omega u^p+\int_\Omega u^p
 \le -C_0\left\|u\right\|_{L^{\frac{n(p+m-1)}{n-2}}}^{p+m-1}
 +C\int_\Omega u^p|\Delta v|+C.
$$
Multiplying this differential inequality by $e^t$ and integrating
over $(0,t)$, we obtain
\begin{equation}\label{Lp est 1}
 \int_\Omega u^p
 \le -C_0\int_0^t e^{s-t}\|u\|_{L^{\frac{n(p+m-1)}{n-2}}}^{p+m-1}\,ds
 +C\int_0^t e^{s-t}\int_\Omega u^p|\Delta v|\,dx\,ds+C.
\end{equation}

In Proposition~\ref{prop-u-n_>=4}, the elliptic equation for $v$
provides a pointwise-in-time estimate of $\Delta v$ in terms of $uw$.
Here, however, $v$ satisfies
$$
 v_t=\Delta v-v+uw,
$$
so no analogous pointwise estimate is available. We therefore estimate
the time-integrated chemotactic term
$$
 \int_0^t e^{s-t}\int_\Omega u^p|\Delta v|\,dx\,ds
$$
directly.

Applying H\"older's inequality first in space and then in time, we
obtain
\begin{align}
 &\int_0^t e^{s-t}\int_\Omega u^p|\Delta v|\,dx\,ds\notag\\
 \le  & \ \int_0^t e^{s-t}  \|u\|_{L^{p \sigma}}^p\|\Delta v\|_{L^{\sigma'}}  \,ds    \notag\\
 \le & \ \left( \int_0^t e^{s-t}  \|u\|_{L^{p\sigma}}^{p+m-1}\,ds \right)^{\frac{p}{p+m-1}}
  \left(
 \int_0^t e^{s-t}
 \|\Delta v\|_{L^{\sigma'}}^{\frac{p+m-1}{m-1}}\,ds
 \right)^{\frac{m-1}{p+m-1}},
 \label{prop-pf-Delta v-1}
\end{align}
where we choose $\sigma$ so that $p\sigma$ coincides with the exponent
in the negative term of \eqref{Lp est 1}:
\begin{equation}\label{def sigma}
\frac{1}{\sigma} = \frac{p}{p+m-1}\frac{n-2}{n},\ \  \textrm{i.e.}, \ \ p\sigma =\frac{n(p+m-1)}{n-2},\ \ \frac{1}{\sigma'}= 1-\frac{1}{\sigma}.
\end{equation}
For $n\ge4$, $m>1+\frac n4$, and $p>1$, relation
\eqref{def sigma} ensures that $\sigma,\sigma'>1$.

Next, we apply Lemma~\ref{max_sobolev_regularity} to the equation for
$v$, with spatial exponent $\sigma'$ and temporal exponent
$\frac{p+m-1}{m-1}$, to obtain
\begin{equation}\label{prop-pf-parabolic}
\int_0^t e^{s-t}
 \|\Delta v\|_{L^{\sigma'}}^{\frac{p+m-1}{m-1}}\,ds \le C\int_0^t e^{s-t}
 \|uw\|_{L^{\sigma'}}^{\frac{p+m-1}{m-1}}\,ds +C.
\end{equation}
To estimate the source term $uw$, define $\beta$ by
$$
 \frac1{\sigma'}=\frac1{p\sigma}+\frac1\beta.
$$
Using \eqref{def sigma}, we compute
$$
 \frac{1}{\beta}
 =1-\frac{p+1}{p\sigma}
 =1-\frac{p+1}{p+m-1}\frac{n-2}{n}.
$$
In particular, $1<\beta<\frac n2$. We fix $\gamma>\frac n2$.
Moreover,
$$
 \lim_{p\to\infty}\frac1\beta
 =\frac2n
 \le\frac{n-2}{n}<\frac1{p_0},
$$
so there exists $P_1>1$, depending only on $n$, $m$, and $p_0$,
such that $\gamma>\beta>p_0$ whenever $p>P_1$. We may therefore
choose $\theta\in(0,1)$ uniquely by
$$
 \frac1\beta=\frac{1-\theta}{p_0}+\frac\theta\gamma.
$$
H\"older's inequality, interpolation, and Lemma~\ref{lp_w_kappa1}
then give
\begin{equation}\label{prop-pf-holder}
 \|uw\|_{L^{\sigma'}}
 \le \|u\|_{L^{p\sigma}}\|w\|_{L^\beta}
 \le \|u\|_{L^{p\sigma}}
      \|w\|_{L^{p_0}}^{1-\theta}\|w\|_{L^\gamma}^{\theta}
 \le C\|u\|_{L^{p\sigma}}\|w\|_{L^\gamma}^{\theta}.
\end{equation}
It follows from \eqref{prop-pf-parabolic} and
\eqref{prop-pf-holder} that
$$
\int_0^t e^{s-t}
 \|\Delta v\|_{L^{\sigma'}}^{\frac{p+m-1}{m-1}}\,ds
 \le C\int_0^t e^{s-t}
 \|u\|_{L^{p\sigma}}^{\frac{p+m-1}{m-1}} \|w\|_{L^\gamma}^{\theta\frac{p+m-1}{m-1}}\,ds +C,
$$
Since $n\ge4$ and $m>1+\frac n4$, we have $m>2$, and hence
$$
 \frac{p+m-1}{p+1}>1,
 \qquad
 k:=\frac{p+m-1}{m-2}>1.
$$
These two exponents are conjugate. Together with
\eqref{prop-pf-Delta v-1}, the preceding estimate yields
\begin{align}
 &\int_0^t e^{s-t}\int_\Omega u^p|\Delta v|\,dx\,ds\notag\\
 &\quad\le C\left(\int_0^t e^{s-t}
 \|u\|_{L^{p\sigma}}^{p+m-1}\,ds\right)^{\frac{p}{p+m-1}}
 \left(
 \int_0^t e^{s-t}
 \|u\|_{L^{p\sigma}}^{\frac{p+m-1}{m-1}} \|w\|_{L^\gamma}^{\theta\frac{p+m-1}{m-1}}\,ds
 \right)^{\frac{m-1}{p+m-1}}\notag\\
 &\qquad{}+C\left(\int_0^t e^{s-t}
 \|u\|_{L^{p\sigma}}^{p+m-1}\,ds\right)^{\frac{p}{p+m-1}}\notag\\
 &\quad\le C\left(\int_0^t e^{s-t}
 \|u\|_{L^{p\sigma}}^{p+m-1}\,ds\right)^{\frac{p+1}{p+m-1}}
 \left(\int_0^t e^{s-t}\|w\|_{L^\gamma}^{\theta k}\,ds
 \right)^{\frac{m-2}{p+m-1}}\notag\\
 &\qquad{}+C\left(\int_0^t e^{s-t}
 \|u\|_{L^{p\sigma}}^{p+m-1}\,ds\right)^{\frac{p}{p+m-1}}\notag\\
 &\quad= C\left(\int_0^t e^{s-t}
 \|u\|_{L^{p\sigma}}^{p+m-1}\,ds\right)^{1-\frac1k}
 \left(\int_0^t e^{s-t}\|w\|_{L^\gamma}^{\theta k}\,ds
 \right)^{\frac1k}\notag\\
 &\qquad{}+C\left(\int_0^t e^{s-t}
 \|u\|_{L^{p\sigma}}^{p+m-1}\,ds\right)^{\frac{p}{p+m-1}}.
 \label{prop-pf-Delta v-2}
\end{align}
Here we used H\"older's inequality in the second step and the definition
of $k$ in the last step. Consequently,
\eqref{Lp est 1} and \eqref{prop-pf-Delta v-2} yield
\begin{align}
 \int_\Omega u^p
 \le & -C_0\int_0^t e^{s-t}\|u\|_{L^{p\sigma}}^{p+m-1}\,ds\notag\\
  & +   C\left( \int_0^t e^{s-t}  \|u\|_{L^{p\sigma}}^{p+m-1}\,ds \right)^{1-\frac{1}{k}}
 \left(\int_0^t e^{s-t} \|w\|_{L^\gamma}^{\theta k}\,ds
 \right)^{\frac{1}{k}}\notag\\
 & +C\left(
 \int_0^t e^{s-t}
 \|u\|_{L^{p\sigma}}^{p+m-1}\,ds
 \right)^{\frac{p}{p+m-1}}+ C.
 \label{Lp est 2}
\end{align}

It remains to control
$$
 \int_0^t e^{s-t}\|w\|_{L^\gamma}^{\theta k}\,ds.
$$
As in \eqref{lm-pf-rho}, define $\rho$ by
$$
 \frac1\rho=\frac1{\sigma'}-\frac2n>0.
$$
If
$$
 p>\max\left\{P_1,\frac{(n-4)(m-1)}2\right\},
$$
then
$$
 \frac1\rho-\frac2n=\frac1{\sigma'}-\frac4n<0.
$$
Together with $\gamma>\beta>p_0$, this ensures the existence of a
unique $\alpha\in(0,1)$ satisfying
$$
 \frac1\gamma
 =\alpha\left(\frac1\rho-\frac2n\right)
  +\frac{1-\alpha}{p_0}.
$$
The parameters $p_0$, $\alpha$, $\theta$, and $k$ are now defined
exactly as in Proposition~\ref{prop-u-n_>=4}. Hence the final parameter
verification in that proof shows that, for all sufficiently large $p$,
$$
 \alpha\theta k>1
 \qquad\text{and}\qquad
 \frac{\alpha\theta}{1-\alpha\theta}k<p+m-1.
$$
{Applying Lemma~\ref{lem:GN}(ii) exactly as in the derivation of
\eqref{lm-pf-higher-w-1} in Proposition~\ref{prop-u-n_>=4}, we obtain}
\begin{align}\label{prop-pf-higher-w-1-parabolic}
\int_0^t e^{s-t}\|w\|_{L^\gamma}^{\theta k}\,ds
 \le   C\int_0^t e^{s-t}\|v\|_{L^\rho}^{\alpha\theta k}\,ds+C,
\end{align}
where we have used the first of these conditions, namely
\begin{equation*}
\alpha\theta k=\alpha\theta \cdot\frac{p+m-1}{m-2}>1.
\end{equation*}
Moreover, the Sobolev embedding, the Neumann elliptic estimate, and
the uniform $L^1$-bound for $v$ yield
$$
 \|v\|_{L^\rho}\le C\|\Delta v\|_{L^{\sigma'}}+C.
$$
Substituting this estimate into
\eqref{prop-pf-higher-w-1-parabolic} and applying
Lemma~\ref{max_sobolev_regularity} once more to
$$
v_t=\Delta v-v+uw,
$$
with time exponent $r=\alpha\theta k$, we obtain
\begin{align*}
 \int_0^t e^{s-t}\|w\|_{L^\gamma}^{\theta k}\,ds
 &\le C\int_0^t e^{s-t}
       \|\Delta v\|_{L^{\sigma'}}^{\alpha\theta k}\,ds+C \\
 &\le C\int_0^t e^{s-t}
       \|uw\|_{L^{\sigma'}}^{\alpha\theta k}\,ds+C.
\end{align*}
Repeating the derivation of \eqref{lm-pf-higher-w-2}, including the
absorption of the corresponding $w$-term, we obtain
$$
 \int_0^t e^{s-t}\|w\|_{L^\gamma}^{\theta k}\,ds
 \le C\int_0^t e^{s-t}
 \|u\|_{L^{p\sigma}}^{\frac{\alpha\theta}{1-\alpha\theta}k}\,ds+C.
$$
Substituting this estimate into \eqref{Lp est 2}, we find
\begin{align*}
 \int_\Omega u^p
 \le & -C_0\int_0^t e^{s-t}\|u\|_{L^{p\sigma}}^{p+m-1}\,ds\\
  & +   C\left( \int_0^t e^{s-t}  \|u\|_{L^{p\sigma}}^{p+m-1}\,ds \right)^{1-\frac{1}{k}}
 \left(\int_0^t e^{s-t} \|u\|_{L^{p\sigma}}^{\frac{\alpha\theta}{1-\alpha\theta}k}\,ds+1 \right)^{\frac{1}{k}}\\
 & +C\left(
 \int_0^t e^{s-t}
 \|u\|_{L^{p\sigma}}^{p+m-1}\,ds
 \right)^{\frac{p}{p+m-1}}+ C.
\end{align*}
The second parameter inequality above ensures that
$$
 \frac{\alpha\theta}{1-\alpha\theta}k<p+m-1,
$$
whereas $\frac{p}{p+m-1}<1$. Young's inequality therefore allows both
positive contributions to be absorbed by the negative term. We conclude
that
$$
 \int_\Omega u^p(\cdot,t)\le C
 \qquad\text{for all }t\in(0,T_{\max}),
$$
which proves the proposition.
\end{proof}

\section*{Data availability}
No data were used for the research described in this article.

{
\section*{Acknowledgments}
\sloppy
X. Bai is supported by the National Natural Science Foundation of China
(Nos. 12671249 and 12271437) and the Guangdong Basic and Applied Basic
Research Foundation (No. 2024A1515010189). F. Li is supported by the
National Natural Science Foundation of China (No. 12371213).

}

\end{document}